\documentclass[12pt]{article}

\usepackage{graphicx}
\usepackage{amsthm}
\usepackage{amsmath,amsthm,amssymb,latexsym}
\usepackage{a4}
\usepackage{xspace}

\usepackage{setspace}

\newtheorem{lemma}{Lemma}[section]
\newtheorem{theorem}{Theorem}[section]
\newtheorem{example}{Example}[section]

\newtheorem{corollary}{Corollary}[section]

\theoremstyle{plain}
\newtheorem{definition}{Definition}[section]
\theoremstyle{remark}
\numberwithin{equation}{section}

\def\PO{{POSS \xspace}}

\def\G0{X^{>}}
\def\Gd{Y}
\def\VMm{{\rm\textsf{V-minmax}}}

\def\nexists{\,\,\,/\!\!\!\exists\,\,}

\def\|{\,\,\,|\,\,\,}

\usepackage{color}

\begin{document}

\title{On two solution concepts in a class of multicriteria games: minimax and Pareto optimal security payoffs
}


\author{Justo Puerto$^1$,
        Federico Perea$^2$ \\
        \\
           $^1$ IMUS, Universidad de Sevilla, Spain (puerto@us.es) \\
           $^2$ Instituto Tecnol\'ogico de Inform\'atica, \\
           Universitat Polit\`ecnica de Val\`encia, Spain (perea@eio.upv.es)
}



\date{}

\maketitle

\begin{abstract}
In  this  paper  we compare two solution concepts for general multicriteria zero-sum
matrix games: minimax and Pareto-optimal security payoff vectors. We characterize the two criteria based on properties similar to the ones that have been used
in the corresponding counterparts in the single  criterion  case,
although they need to be complemented with two new consistency
properties. Whereas in standard single criterion games minimax and optimal security payoffs coincide, whenever we have
multiple criteria these  two solution concepts differ. We provide
explanations for the common roots  of these two  concepts and highlight the intrinsic differences between them.

\textbf{Keywords:} Multicriteria games ; payoffs ; characterization
\end{abstract}

\section{Introduction}

Multicriteria zero-sum matrix (MZSM) games are an extension of the standard
two-person zero-sum games, introduced by von Neumann
\cite{newmann28}, that allow to handle simultaneous confrontation of
two rational agents in several scenarios. The first game theoretical
characterizations of minimax values in MZSM games is due to Shapley
\cite{shapley59}. In its introductory note, F.D.
Rigby remarks the importance of these games:
\textit{``The topic of games with vector payoffs is one which could
be expected to attract attention on the basis of its intrinsic
interest''}.  Nevertheless, the development of the theory of multicriteria games has not been as successful as expected. There is a number of references in the literature, although its intrinsic difficulty, mainly due to the lack of total orderings among players' payoffs, has diminished the interest of researchers (see for instance \cite{BTA89}, \cite{fernandez96}, \cite{ghose89}, \cite{ghose91}, \cite{nieuwenhuis83}, \cite{ppt07}, \cite{shapley59}, \cite{Voor99} and \cite{Voorequi} and the references therein). In spite of that, the goal of further developing the analysis of multicriteria games should not be forgotten. In fact, each competitive situation that can be modeled as a scalar zero--sum game has its counterpart as a
multicriteria zero--sum game when more than one scenario has to be
compared simultaneously. Moreover, different scenarios do not need to have the same set (number) of strategies which makes the analysis even more challenging, as in some scenarios new strategies may appear.  In these situations, applying an extension of the minimax rationale would imply to use the same strategy
 in the different scenarios. Thus, conflicting interests
appear not only between different decision-makers but also within each
individual, due to the different criteria they may have. For instance, the production policies of two firms which
are competing for a market can be seen as a scalar game. However,
when they compete simultaneously in several markets and the returns
in each one of them cannot be aggregated, the multicriteria approach
{naturally leads} to a multicriteria game. The main criticism made to
this approach is its difficulty to be applied because in most cases the
solutions (values) are not unique. Therefore, new solution
concepts have been proposed and compared with the
existing ones.

One of the most attractive alternatives to minimax payoff vectors is
the concept of Pareto-optimal security level vectors (POSLV). Pareto-optimal
security strategies (POSS) were introduced by Ghose and Prasad
\cite{ghose89} as a solution concept in multicriteria zero-sum
matrix games, extending the idea of security level strategies to more
than one criterion where it is allowed to use different strategies in each scenario.
Ghose \cite{ghose91} characterized this solution
concept as minimax strategies in a serial weighted zero-sum game,
whereas Voorneveld \cite{Voor99} gave an alternative
characterization of these strategies as minimax strategies of an
\textit{amalgamated} game (see \cite{BGPT96} for the concept of
amalgamation of games). Alternatively, Fern\'andez and Puerto
\cite{fernandez96} provided a way to jointly determine POSS and their corresponding set of payoffs by solving a
certain multiobjective linear program.

The main goal of this paper is to highlight the connections and
differences between minimax and Pareto-optimal security strategies in
multicriteria zero-sum matrix games.

{The rest of the paper is organized as follows. In Section \ref{sec: lit review} we give some references on related work. In Section \ref{sec: model} we present
 the basic definitions of the multicriteria games, in which we will define the two solution concepts to be later studied.  In Section \ref{sec:minmax} we formally define the minimax solution concept, and characterize it using a number of axioms. Section \ref{sec:def_POSS} is devoted to
the definition and characterizations of the set of Pareto-optimal security
level vectors. Section \ref{sec: independence} is devoted to analyzing the {pairwise} logical independence of the properties used in the presented characterizations.} {Next in Section \ref{s:6} we analyze and prove the relationships between the two solution concepts.} The paper
ends with some conclusions drawn from the results in the paper.

\section{Literature review}\label{sec: lit review}
To provide the necessary properties for characterizing these
two procedures we start by reviewing the literature regarding axiomatizations of
solution concepts in zero-sum matrix games.

The first characterization of the value of a zero-sum matrix game is
due to Vilkas \cite{vil}. Later, Tijs \cite{Tijs81} addressed the
same problem, whereas more recently  Norde and Voorneveld
\cite{norde04} and Carpente et al. \cite{car}  provided different
axiomatizations for the value of standard zero-sum matrix games. Hart et al. \cite{Hart94} give a Bayesian foundation of {minimax} values.

Although there exist axiomatic foundations for the minimax
value of one-criteria zero-sum games, there is no characterization for the set
of minimax values in the multicriteria version. On the
other hand, there are several axiomatizations of  Nash equilibrium
strategies (see e.g. Peleg and Tijs \cite{pel96}, Peleg et al.
\cite{pel96bis} and Norde et al. \cite{norde96}) some of these carry
over to multicriteria games, as shown in  Borm et al. \cite{BTA89}
and Voorneveld et al. \cite{Voorequi}. Nevertheless, these are axiomatizations for strategy sets rather than for payoffs, they use different sets of axioms, and they do not clearly show
the relationship between the two solutions.

In standard single criterion games, minimax and optimal security
payoffs coincide. Nevertheless, whenever we have  multiple criteria, which in turns may imply to have a different set (and number) of strategies in each component game,
these two concepts differ: uniqueness and coincidence with security
payoffs are no longer satisfied. This fact raises the question of
which are the common roots and which are the main differences between these two solution concepts. We will answer this question by providing characterizations that show the common properties and
 the differences between them.

We here characterize extensions of minimax and Pareto-optimal security payoffs to general
multicriteria zero-sum matrix games. Our approach uses classical properties in
game theory and decision theory (objectivity, column dominance, column
elimination, row dominance, row elimination and consistency). The contribution of this paper is twofold:
\begin{itemize}
\item we characterize solution sets rather than single value solutions which makes the
analysis more involved. In this regard, we will be interested in finding the largest set of payoffs compatible with the proposed properties (see e.g. Calleja et al. \cite{CRT08} or Gerard-Varet and Zamir \cite{GVZ}).
\item we use a new consistency property that deals with the persistence of any solution payoff of a multicriteria game, with given dimension on the space of criteria,  on some lower dimension multicriteria game. 
The difference between this new consistency property and the traditional one is the way in which solutions for a game with $k$ criteria transform into solutions for a game with $k-1$ criteria. Extended minimax payoff vectors in a multicriteria game with $k$-criteria can be converted to extended minimax payoff vectors in a new $(k-1)$-criteria game that makes a convex combination of two of the original payoff matrices. However, Pareto-optimal security payoff vectors become solutions of a game with $k-1$ criteria but with an amalgamation of strategies from two of the previous matrices. This difference is crucial and distinguishes the two solution concepts.
\end{itemize}

\section{Definitions}\label{sec: model}
 Let us begin by recalling the concept of two-person zero-sum games.
\begin{definition}
A scalar two-person zero sum game is characterized by a payoff matrix $A \in \mathbb{R}^{m \times n}$, such that if player I (the row player who wants to minimize payoffs) plays his strategy $i$ and player II (the column player who wants to maximize payoffs) plays his strategy $j$, then the row player`s payoff is $a_{ij}$ and the column player's is $-a_{ij}$, for $i=1,...,m, j=1,...,n$. The set of (mixed) strategies for player I is
$$X_m = \{ x\in \mathbb{R}^m :\sum _{i=1}^m x_i =1;\ x_i\ge 0, \ \ \
i=1, ... ,m \},$$
whereas the set of (mixed) strategies for player II is

$$Y_n = \{ y \in \mathbb{R}^n :\sum _{j=1}^n y_j =1;\ y_j\ge 0, \ \ \
j=1, ... ,n \}.$$

Note that if player I plays $x \in X_m$ and player II plays $y \in Y_m$, the expected payoff of the game is $x ^t A y \in \mathbb{R}$. We define the value of a game $A$ as $val(A):=\max_{y \in Y_n} \min_{x \in X_m} x^t A y =  \min_{x \in X_m} \max_{y \in Y_n} x^t A y$, which always exists (see \cite{newmann28}).
\end{definition}

{The traditional multicriteria approach (see Definition \ref{def: multicriteriaclassic}) assumes that the payoffs of the players are vectors instead of scalars.}
\begin{definition}\label{def: multicriteriaclassic}
{Consider a payoff matrix $A =(A(1),...,A(k)) \in \mathbb{R}^{m \times n \times k},$ with $A(\ell) \in \mathbb{R}^{m \times n} \ \forall \ \ell = 1,...,k$. If player I plays $x \in X_m$ and player II plays $y \in Y_n$, the expected payoff for player I is $x ^t A y = (x^t A(1) y,...,x^tA(k) y) \in \mathbb{R}^k$, and player II gets the opposite.} Let $\mathcal{D}_1$ denote the class of games defined by a matrix $A \in \mathbb{R}^{m \times n \times k}$, and the strategy sets $X_m$ and $Y_n$ for players I and II, respectively.
\end{definition}
Note that a game in $\mathcal{D}_1$ is uniquely characterized by its payoff matrix, since the strategy sets only depend on the dimensions of the matrix. We therefore may identify a game in this class by its payoff matrix.

To compare the two solution concepts that are the goal of this paper, we analyze them in an extension of the general class of  multicriteria two-person zero sum games $\mathcal{D}_1$ just introduced, defined as follows:

\begin{definition}
Consider a payoff matrix $A=(A(1), ... ,A(k))$  with  $A(\ell) \in \mathbb{R}^{m \times n_{\ell}}$ for $\ell=1, ... ,k$. Let $G_k=\bigcup _{m,n_1, ... ,n_k\in \mathbb{N}} \mathbb{R}^{m\times n_1}\times ... \times \mathbb{R}^{m\times n_k}$ be the set of all such $k$-criteria matrices. Define the strategy spaces for players I and II as $X_m$ and,
$$\begin{array}{rl}
Y_{(n_1\ldots  n_k)}=&\{ y=(y(1),\ldots,y(k))\in \mathbb{R}^{n_1}\times \ldots \times \mathbb{R}^{n_k} :\sum _{j=1}^{n_{\ell}} y(\ell)_j =1;\ y(\ell)_j\ge 0, \\
 & j=1,\ldots ,n_{\ell},\; \ell=1\ldots k \}. \\
\end{array}$$

Note that player I has to play the same strategy in all $k$ scenarios, whereas player II can choose different strategies in different scenarios. We remark that the pure strategies for both players are the extreme points of $X_m $ and $Y_{(n_1\ldots n_k)} $.

Choosing the strategy $x\in X_m $ for player I and the strategy $y\in Y_{(n_1\ldots n_k)}$ for player II implies that the payoff of player I is
\begin{equation}
v(x,y)=x^tAy=(v_1(x,y(1)),\ldots ,v_k(x,y(k))),
\end{equation}
where
\begin{equation}
v_{\ell}(x,y({\ell}))=x^tA({\ell})y({\ell}) \qquad {\ell}=1,\ldots ,k.
\end{equation}
Let $\mathcal{D}_2$ denote the class of games defined by a tridimensional payoff matrix $A = (A(1),...,A(k)) \in \mathbb{R}^{m\times n_1}\times \ldots \times \mathbb{R}^{m\times n_k}$, and the strategy sets $X_m$ and $Y_{(n_1,...,n_k)}$ for players I and II, respectively.
\end{definition}
Note that a game in $\mathcal{D}_2$ is uniquely characterized by its payoff matrix, since the strategy sets only depend on the dimensions of the matrix. We therefore may identify a game in this class by its payoff matrix.

For the sake of notation we will refer to $\mathbb{R}^{m\times n_1}\times \ldots \times \mathbb{R}^{m\times n_k}$ as $\mathbb{R}^{m \times (n_1,...,n_k)}$.
The following example illustrates this class of games.
\begin{example}\label{ex:gameD2}
Let $A$ and $B$ be two companies that form a duopoly competing in the same sector. The payoffs of this game will be monetary benefit and public image, measured in units that are not easily quantifiable economically. Company $A$ has to face one decision: whether or not to invest in advertising (strategies $Y$ and $N$, respectively). Company $B$ has to face two decisions: whether or not to invest in advertising (strategies $Y$ and $N$, respectively), and what to do about the polluting emissions its factory produces. Three different strategies are possible in this scenario: increase, leave as it is, decrease (denoted by $I,L,D$).

Both companies know that if the two of them invest in advertising or none of them does, then their monetary benefit does not increase nor decrease. On the contrary, if one of them invests and the other does not, the company investing will have one extra unit of benefit and the other company one less unit of benefit.

If $B$ increases the polluting emissions and $A$ invests in publicity, $A$ will use this fact in its campaign and the public image of $B$ will deteriorate and that of $A$ will improve by two units. In case $A$ would not invest in advertising, the extra emissions will be somehow found out and the public image of $B$ will deteriorate and that of $A$ will improve by 1 unit. If $B$ decreases its emissions and $A$ does not invest in publicity, the public image of $B$ will improve and that of $A$ will deteriorate by one unit. An advertising campaign of $A$ will compensate this fact and the improvement/deterioration in public image will be of 0.5 units only.

This situation can be modeled as a game in $\mathcal{D}_2$, with $m=k=n_1=2, n_2 = 3$ and the payoff matrix
$ A = (A(1),A(2))$, where $A(1)$ and $A(2)$ are detailed in the next matrix:

$$  A = (A(1),A(2)) =
\left(
  \begin{array}{cc}
    \left(
      \begin{array}{cc}
        0 & -1 \\
        1 & 0 \\
      \end{array}
    \right) ,
     &
     \left(
       \begin{array}{ccc}
         -2 & 0 & 0.5 \\
         -1 & 0 & 1 \\
       \end{array}
     \right)
     \\
  \end{array}
\right)
$$
Note that, in order to be consistent with the fact that player I is the minimizer and player II is the maximizer, the payoffs described before have been multiplied times -1 when building matrix $A$.
To illustrate, two possible such strategies are:
\begin{enumerate}
\item $x=(1,0), y=((0,1),(0,1,0))$. They are two pure strategies that yield an expected value of the game equal to $(0,0)$ (the payoff for player I in each of the two criteria, player II gets the opposite).
\item $x= (0.5,0.5), y = ((0.25,0.75),(0.5,0,0.5))$. They are two mixed strategies that yield an expected value equal to $(-0.25,-0.375)$.
\end{enumerate}
\end{example}

Note that in this framework we identify a multicriteria game with an element in the cartesian product of the space of matrices. At times, the scalar zero-sum game defined by the matrix $A(\ell)$, $\ell=1\ldots k$, will be called the $\ell$-component game of the multicriteria game.

In the sequel, the transpose operator $^t$ will be omitted when its use
is clear. Given a matrix $B \in \mathbb{R}^k$, we will adopt the following notation:
\begin{itemize}
\item $B^{i\cdot}$ refers to the $i$-th row
of $B$.
\item $B^{\cdot j}$ refers to the $j$-th column of $B$.
\end{itemize}
We will use the following ordering.
\begin{definition}
Given two vectors $a,b \in \mathbb{R}^k$, $a \leq b$ if $a_{\ell} \leq b_{\ell}$ for $\ell = 1,..., k$, and $a \lneq b$ if $a \leq b$ and $a \neq b$.

$v-\min$ ($v-\max$) stands for the set of minimal (maximal) elements with respect to the componentwise order of $\mathbb{R}^k$.
\end{definition}

\section{{Minimax} in $\mathcal{D}_1$}\label{sec:minmax}

The multicriteria extension of the concept of minimax payoff
looses some of the interesting properties shown in the scalar case:
uniqueness and coincidence with security payoffs. In spite of that,
it is still possible to establish the existence of such strategies under rather standard hypothesis; for instance if the strategy set is compact the set of $minimax$ payoff
vectors is non-empty \cite{nieuwenhuis83}.

The rationale behind minimax strategies is that each player uses the same strategy in all the  $k$-component games looking for all the non-dominated vector valued payoffs. This rationale is possible in the domain $\mathcal D_1$, and so we shall consider multicriteria games defined on this domain to analyze multicriteria minimax payoff vectors.

Therefore, {for each strategy $x \in X_m$} of player I let
\begin{equation}\label{eq:v-max}
v-w_1(x) :=v- \max_{y \in {Y_{n}}} (x A(1) y, \ldots, x A(k) y ).
\end{equation}
Now, among those strategies that give maximal payoffs in the above
problems, player I will choose the strategies with the \textit{best}
payoff. (Note that \textit{best} (\textit{worst}) must be understood as finding the
set of minimal (maximal) elements in the componentwise ordering of
$\mathbb{R}^k$.) We now define the set of minimax payoffs in the  $k$-criteria zero-sum game domain $\mathcal D_1$.

\begin{definition}
The set of \textit{extended minimax} payoffs of the game $A \in \mathcal{D}_1$ is given by:
\begin{equation}{\small \label{def:minimax}\mbox{\VMm}_k (A(1),\ldots,A(k)) = v-\min \hspace*{-0.2cm} \bigcup_{x \in X_m} \hspace*{-0.1cm}
v-w_1(x).}
\end{equation}
\end{definition}

According to the above expression the extended minimax payoff vectors {are
the non-dominated vectors} obtained from the solutions to all
the vector valued maximization problems \eqref{eq:v-max} for all
$x\in X_m$.

Clearly, if $B\in \mathbb{R}^{m\times p}$ then $\VMm_1(B)=val(B)$.

After defining the  extended minimax payoff vectors, the concept of  extended minimax strategy naturally follows.
\begin{definition}
An \textit{extended  minimax} strategy of player I is any strategy attaining an extended
minimax payoff vector. (Similarly, one can define  extended \textit{maximin} strategies of
player II.)
\end{definition}

The first result about extended minimax payoff vectors in multicriteria games
was given by Shapley \cite{shapley59}, who provides a simple way for
finding them by solving zero-sum scalar games with payoff matrix
$A(\alpha) =  \sum_{\ell=1}^k \alpha_{\ell} A(\ell)$, a positive linear
combination of matrices $ A({\ell}),\: {\ell}=1,\ldots ,k$.

\begin{theorem}[Adapted from \cite{shapley59}]\label{new}
Let $z^*$ be an extended  minimax value for the multicriteria game with matrix
$A=(A(1),\ldots,A(k))$ then there exists  $\alpha\in X^>
_k:=\{x\in\mathbb{R}^k: \sum_{{\ell}=1}^k x_{\ell}=1,\; x_i>0, \; \forall i\}$
such that $\sum_{{\ell}=1}^k \alpha_{\ell} z^*_{\ell} = \mbox{ val}\,\,
(\sum_{{\ell}=1}^k \alpha_{\ell} A({\ell})).$

\noindent Conversely, let $z^*(\alpha)$ be the minimax value of
$(\sum_{\ell=1}^k \alpha_{\ell} A(\ell))$, then there exists an extended  minimax payoff
vector $z^* \in \VMm_k(A)$, being  $A=(A(1),\ldots,A(k))$,
satisfying $\sum_{{\ell}=1}^k \alpha_{\ell} z^*_{\ell} = z^*(\alpha)$.
\end{theorem}

\subsection{A characterization of the  set of extended minimax payoff vectors\label{sec: minmax}}
We begin this section by introducing the axioms that will characterize the set of minimax payoff vectors.

Let $\{f_k\}_{k\ge 1}$ be a family of point-to-set maps (correspondences) defined as
\begin{eqnarray*}
 f_k  : & G_k &
\longrightarrow 2^{\mathbb{R}^k} \\
& A = (A(1),\ldots,A(k))& \longrightarrow  f_k(A).
\end{eqnarray*}

The axioms needed are:
\begin{itemize}

\item[A.0] {\bf Objectivity.} For any $z \in \mathbb{R}^k$ $f_k(z )= z$.

\item[A.1] {\bf Monotonicity.} For any $A,\overline{A} \in \mathbb{R}^{m \times (n_1,\ldots,n_k)}$ such that
$\overline{A}\leq A$, $f_k(\overline{A})\subseteq f_k(A) +
\mathbb{R}^k_{-}$

\item[A.2] {\bf Column dominance.}  Let $A_{c({\ell})}$ be the matrix that results from $A$ after adding to
the matrix $A({\ell})$ a new column which is dominated by a convex
combination of its columns. Then $ f_k(A_{c({\ell})}) = f_k(A).$

\item[A.3] {\bf Column elimination.}  Let $A_{-c({\ell})}$ be the matrix that results from $A$ after
removing column $c(\ell)$ from  the matrix $A({\ell})$. Then $f_k(A_{-c({\ell})}) \subseteq f_k(A) +\mathbb{R}^k_{-}$.

\item[A.4] {\bf Row dominance.}  Let $A_r$ be the matrix that results from $A$ after adding
 a new row which is dominated by a convex combination of its rows. Then $ f_k(A_r) =
f_k(A).$

\item[A.5] {\bf Row elimination.}  Let $A_{-r}$ be the matrix that results from $A$ after
removing row $r$. Then $f_k(A_{-r}) \subseteq f_k(A) + \mathbb{R}^k_{+}.$

\item[A.6]. {\bf Consistency.}  For any $k\geq 2$, if
$z \in f_k(A)$ then there exists $0<\alpha<1$, such that
$$(\alpha z_1 + (1-\alpha) z_2, z_3,\dots,z_k) \in
f_{k-1}(M(\alpha A(1) ,  (1-\alpha) A(2) ), A(3),\ldots, A(k))$$
where $M(A(1),\ldots,A(k))$ is a matrix with $m$ rows, labeled
$i=1,\ldots,m$ {and $\prod_{\ell} n_{\ell}$ columns, labeled $c= ( c(1),\ldots,c(k))$
with $c({\ell}) \in \{ 1,\ldots, n_{\ell}\}$} for each ${\ell}=1,\dots,k$. The entry
in row $i$ and column $c=(c(1),\ldots,c(k))$ of $M(A)$ equals
$\sum_{{\ell}=1}^k A({\ell})_{ic({\ell})}$. (See \cite{Voor99}.)

\item[A.7.] {\bf Linear
consistency.}  For any $k\geq 2$ and $A$ such that $n_{\ell} = n \ \forall \ \ell$, if $z \in f_k(A)$ then there
exists $0<\alpha<1$ such that $(\alpha z_1 + (1-\alpha) z_2,
z_3,\dots,z_k) \in f_{k-1}((\alpha A(1)+(1-\alpha) A(2) ),
A(3),\ldots, A(k))$.
\end{itemize}

Objectivity establishes the evaluation in a trivial situation where the game has $k$ criteria and both players have exactly one action available. Monotonicity states that the set of solution payoff vectors should not decrease, in the componentwise order of $\mathbb{R}^k$, when all the payoff matrices weakly increase. Column (row) dominance states that the set of solution payoff vectors should not change if player II (I) can no longer choose an action, in some of the component games, which is worse for him than a combination of some other actions. Column  elimination states that the set of solution payoff vectors can not increase their values, in the componentwise order of $\mathbb{R}^k$, when some action of player II in some of the component games is removed. Row elimination states that when removing an action of player I the new set of solution payoff vectors must be dominated by the original one.
Consistency  states that any solution outcome of a game with a given dimension, in the criteria space, can be converted into a solution outcome of a new game with lower dimension of an amalgamated game `\`a la' Borm et al. \cite{BGPT96}.
Linear consistency states that any solution outcome provided by this correspondence, with a given dimension in the criteria space, can be converted into a solution outcome of this correspondence with lower dimension by considering a convex combination of two of the original criteria.

The next result is a characterization of the set of extended minimax payoff vectors of any general multicriteria zero-sum game.

\begin{theorem}\label{t:cMm}
The set of extended minimax payoff vectors $\VMm_k$ is the largest (w.r.t. inclusion) map on $\mathcal D_1$, the set of multicriteria zero sum games, that satisfies objectivity, monotonicity,  column dominance for $k=1$, row dominance and linear consistency.
\end{theorem}
\noindent{\bf Proof}: The proof is similar to Theorem \ref{main},
 but using Theorem \ref{new} instead of Theorem \ref{multiples}.
First of all, we check that $\VMm _k$ satisfies the properties.

\begin{description}

\item [Objectivity] It is clear that $\VMm _k$ satisfies A.0.
\item [Monotonicity]
Since $x\ge 0$ then $xA({\ell})\ge x\bar{A}({\ell})$, for all ${\ell}=1,\ldots,k$.
Hence, for all  $y\in Y_{n}^2$, $( xA(1)y,\ldots, xA(k)y)\geq ( x\overline{A}(1)y,\ldots,
x\overline{A}(k)y)$, and the property
follows.

\item [Column dominance for $\mathbf{k=1}$] It is clear that $\VMm _1$ satisfies column dominance for $k=1$ since $\VMm_1(B)=val(B)$ and it is known that the value function, $val(\cdot)$, of a matrix game satisfies this property (see e.g. \cite{car}).
%

\item [Row dominance] Let $ A_r = \left(  \begin{array}{ccc}
   & A&  \\
  b_1 & \ldots & b_k
\end{array}  \right) $, $ b_{\ell}=(b_{\ell}^1,\ldots,b_{\ell}^{n_{\ell}}) \in \mathbb{R}^{n_{\ell}},\,
{\ell}=1,\ldots,k$  such that $b_{\ell} = \sum_{i=1}^m \alpha_i A^{i\cdot}({\ell})$
with  $\sum_{i=1}^m \alpha_i=1$ and $\alpha_i\ge 0$. Take
$\overline{x}=(\overline{x}_1,\ldots,\overline{x}_m,\overline{x}_{m+1})
\in X^1_{m+1}$, then for any $y\in Y_n$
$$\overline{x} \left(\begin{array}{c} A({\ell})\\
b_{\ell}
\end{array} \right) y= \sum_{i=1}^m \sum_{j\in I_{\ell}} (\overline{x}_i+\alpha_i \overline{x}_{m+1}
) a_{ij}({\ell}) y_j= \hat{x} A({\ell}) y,\; \forall \; {\ell}=1,\ldots,k,$$ being
$\hat{x}=((\overline{x}_1+\alpha_1 \overline{x}_{m+1} ), \ldots,
(\overline{x}_n+\alpha_n \overline{x}_{m+1} ))\in X_m$. Hence,
$\VMm_ k (A_r)=\VMm_k(A)$.

\item [Consistency] Assume A.7 is not satisfied. Therefore, there
exists  $z=(z_1,\dots,z_k) \in \VMm_k(A)$ such that for all
$0<\alpha<1$, $(\alpha z_1+ (1-\alpha) z_2,z_3,\ldots,z_k) \notin
\VMm_{k-1}([\alpha A(1)+ (1-\alpha) A(2)], A(3),\ldots,A(k)).$ Hence
by Theorem \ref{new}, it does not exist
$(\beta_1,\beta_3,\ldots,\beta_{k})\in \G0_{k-1}$ such that
\begin{eqnarray*} \beta_1\alpha z_1 + \beta_1 (1-\alpha) z_2+ \beta_3z_3+\ldots+ \beta_k z_k
\in  & val(\beta_1\alpha A(1)+\beta_1 (1-\alpha) A(2)+ \\
 & \beta_3 A(3)+\ldots+ \beta_k A(k)).
\end{eqnarray*}

 However, this contradicts that $z$ is a minimax payoff vector because
  by Theorem \ref{new}, there must exist $(\lambda_1,\ldots,\lambda_k)\in X^>_k$,
 such that $$\sum_{{\ell}=1}^k \lambda_{\ell} z_{\ell} = val ( \lambda_1 A(1)+\ldots+\lambda_k A(k)).$$

 The above contradiction proves that $\VMm_k$ satisfies A.7.
\end{description}

To finish the proof, it is enough to show that if $\{g_k\}_{k\ge 1}$ is an arbitrary family  of point-to-set maps that satisfy A.0, A.1, A.2, A.4 for $k=1$, and A.7, then for any general multicriteria game given by the matrix $A=(A(\ell))_{\ell=1\ldots k}$ with $A(\ell) \in {\mathbb{R}^{m\times n}}$ we get $g_k(A) \subseteq \VMm _k(A) \quad
\forall\; k\geq 1.$

 For $k=1$, the axioms A.0, A.1, A.2 and A.4  characterize the $val(.)$
 function of a matrix game (see Carpente et al.\cite[Theorem 2]{car}). Therefore, $g_1(.)= val(.) = \VMm _1(.).$

 Let $z \in g_k(A), k\ge 2$. Apply A.7 $k$-times to conclude that there {exists
 $\alpha\in \mathbb{R}^k$, $\sum_{{\ell}=1}^k \alpha_{\ell}=1,\; \alpha_{\ell}>0$,} such that
 $$\sum_{{\ell}=1}^k \alpha_{\ell} z_{\ell} = val ( \sum_{{\ell}=1}^k \alpha_{\ell} { A({\ell})}).$$
Then, by Theorem \ref{new}, $z$ is an extended  minimax payoff vector of the
multi-matrix $A$. Hence, $$ g_k(A) \subseteq \VMm _k(A) \quad
\forall A,\; k\geq 1.$$

\hfill$\Box$

{We would like to remark that in our characterization we use maximality with respect to inclusion. This property is rather important when dealing with set-valued functions (correspondences) since it ensures that this is the largest  object (solution concept) satisfying the required game theoretic properties. This approach is not new and  has been already used among others by  Gerard-Varet and Zamir  \cite{GVZ} for characterizing the \textit{`Reasonable set of outcomes'} and Calleja et al. \cite{CRT08} for the \textit{`Aggregate-monotonic core'}.
}

{
The above theorem also implies that $\VMm_k$ is the largest map on $\mathcal{D}_1$ that satisfies Linear Consistency (A7) and that coincides with the value function on standard single criterion matrix games.
Another characterization of extended minimax payoff vectors using a different set of properties is possible. The rationale is to alternatively characterize the $val(\cdot)$ function and then to
apply the consistency construction. This is possible based on
Carpente et al. \cite[Theorem 3]{car}.}

{\begin{theorem} \label{t:ccMm}
The set of extended minimax payoff vectors $VMm_k$ is the largest (w.r.t. inclusion) map on $\mathcal D_1$ that satisfies  objectivity,   column dominance, row
dominance, column elimination, row elimination and linear consistency.
\end{theorem}
The proof runs similarly to that of Theorem \ref{t:cMm} but using
\cite[Theorem 3]{car} instead of \cite[Theorem 2]{car}.}

%
%

\section{POSS in $\mathcal{D}_2$}\label{sec:def_POSS}

Apart from minimax values, our interest also goes into another
solution concept of multicriteria zero-sum matrix games: POSS. This
concept is independent of the notion of equilibrium, so that the
opponents are only taken into account to establish the security
levels for one's own payoff. Therefore this notion does not require to play the same strategy in all the $k$-component games, and thus is defined in the class $\mathcal{D}_2$.

\begin{definition}
Every strategy $x \in X_m $  defines security levels $v^{\ell}_I(x)$ as
the payoffs with respect to each criterion,
 when II  bets to maximize   the criteria \cite{ghose91}. Hence
\begin{eqnarray}
v^{\ell}_I(x) =  \max _{y\in Y_{n_{\ell}}} x A(\ell) y= \max _{y\in Y_{n_{\ell}}} v_{\ell}(x,y), \qquad \ell=1,\ldots ,k,
\label{seg1}
\end{eqnarray}
and the security levels are $k$-tuples denoted by
\begin{eqnarray}
v_I(x) =(v^1_I(x),\ldots ,v^k_I(x)).
\end{eqnarray}
\end{definition}

Whenever the game we are referring to is not obvious we will use the notation $v_I(x,A)$ to specify such matrix $A$.
It must be noted that for a given strategy $x$ for player I, the
security levels $v_I(x)$ might be obtained by different strategies
of  player II .
In \cite{ghose89} the concept of Pareto optimal security strategy (POSS) is defined as follows.
\begin{definition}
A strategy $x^*\in X_m$ is a
Pareto-optimal security strategy for I if and only if there is no $ x\in
X_m$ such that $v_I(x^*) \gneq v_I(x)$. Similarly, one can define \PO for II.

The set of Pareto-optimal security level
vectors is the set of payoffs that can be attained by POSS, and will
be denoted by $\mbox{VPOSS}_k(A(1),\ldots,A(k))$, thus

\begin{eqnarray}
VPOSS_k(A(1),\ldots,A(k))& =& \{ z \in \mathbb{R}^k: z=v_I(x) \mbox{ for some } x\in X_m \mbox{
and }  \nonumber \\
& & \nexists y \in X_m \mbox{ such that } v_I(y)\lneq v_I(x) \}.
\end{eqnarray}
\end{definition}
Note that depending on $k$, the corresponding set is defined on a
different framework space $\mathbb{R}^k,\; k\ge 1.$

The following theorem provides a way to determine all \PO and their security level vectors.
\begin{theorem}(\cite[Theorem 3.1]{fernandez96})   \label{multiple} A strategy $x^*\in X_m $ is a POSS
and $v^*=(v_1^*,\ldots ,v_k^*)$ is its security level vector if and only if
$(v^*,x^*)$ is an efficient {solution to} the problem
$$\begin{array}{clc}
\min & v_1,\ldots ,v_k, & \\
\mbox{s.t.} & xA({\ell})\le (v_{\ell},\ldots ,v_{\ell}), & {\ell}=1,\ldots ,k,\\
& \sum _{i=1}^n x_i=1, &
\\    & \qquad x\ge 0, &v\in \mathbb{R}^k.
\end{array} $$
\end{theorem}

Equivalently, Ghose  (\cite[Theorem 3.3]{ghose91}) and Voorneveld
(\cite[Theorem 3.1]{Voor99}) characterize POSS strategies as minimax
strategies of  particular classes of scalar games.

\begin{theorem}  \label{multiples} (\cite[Theorem 3.3]{ghose91}) A strategy $x^*\in  X_m$ is a POSS
for player I in the multicriteria game  if and only if there exists $ \alpha
\in X^>_k=\{x\in\mathbb{R}^k: \sum_{{\ell}=1}^k x_{\ell}=1,\; x_i>0, \;
\forall i\}$, such that
$$  \max_{(y_1,\dots,y_k)\in \prod_{{\ell}=1}^k\Gd_{n_{\ell}}}\;\;\sum_{{\ell}=1}^k x^* \alpha_{\ell} A({\ell})y_{\ell} = \min_{x \in X_m}
\;\; \max_{(y_1,\dots,y_k)\in \prod_{{\ell}=1}^k\Gd_{n_{\ell}}}\;\;
\sum_{{\ell}=1}^k x \alpha_{\ell} A({\ell})y_{\ell}.$$
\end{theorem}

\begin{theorem} \label{t:Voorn} (\cite[Theorem 3.1]{Voor99})
A strategy $x^*\in X_m$ is a \PO for player I in the multicriteria
matrix game $A=(A(1),\ldots,A(k))$ if and only if there exists a
vector  $\alpha \in X^>_k$ such that $x^*$ is a minimax strategy in
the scalar matrix game  $M(\alpha_1A(1),\ldots,\alpha_kA(k))$ being
 $M(\alpha_1A(1),\ldots,\alpha_kA(k))$ a matrix with $m$ rows, labeled
$i=1,\ldots,m$ and $n^k$ columns, labeled $c= ( c(1),\ldots,c(k))$
with $c({\ell}) \in \{ 1,\ldots, n\}$ for each ${\ell}=1,\dots,k$. The entry
in row $i$ and column $c=(c(1),\ldots,c(k))$ of $M(\alpha A)$ equals
 $\mathbf{\sum_{{\ell}=1}^k \alpha_{\ell}A({\ell})_{ic({\ell})}}$.
\end{theorem}

The reader may note that in the scalar case,  $\mbox{VPOSS}_1(B)= \VMm_1(B)=val(B)$,
and theorems  \ref{multiple},  \ref{multiples} and \ref{t:Voorn}
coincide.

\subsection{A characterization of Pareto-optimal security payoffs  \label{sec: poss}}

In this section, we characterize the set of
Pareto-optimal security payoffs  defined on a general multicriteria
two--person zero-sum game in $\mathcal D_2$. Thus, we try to identify a map (solution concept)
that associates to any array of $k$ matrices {with the same number of rows} a set of vectors in $\mathbb{R}^k, k\ge 1.$


Using the above set of properties, we can obtain a
characterization of the entire set of Pareto-optimal security payoff vectors as the maximal (in the inclusion sense) point-to-set map that satisfies $A.0$, $A.1$, $A.2$, $A.4$ and $A.6$.

\begin{theorem}\label{main} The set of Pareto-optimal security level vectors $VPOSS_k$ is the largest (w.r.t. inclusion) map on $\mathcal D_2$, that satisfies objectivity, monotonicity,  column dominance,
row dominance and consistency. 

\end{theorem}
\noindent{\bf Proof}: First of all, we check that $VPOSS_k$
satisfies these properties. 

\begin{description}

\item [A.0 Objectivity] It is clear that $VPOSS_k$ satisfies A.0.
\item [A.1 Monotonicity] The security level vectors for the strategy $x$ with respect to $A$
$$
\begin{array}{rl}
v_I(x, A) & = (\max_{y \in Y_{n_1}} x A(1) y,..., \max_{y \in Y_{n_k}} x A(k)) \\
 & = (\max_{y \in ext\{Y_{n_1}\}} x A(1) y,..., \max_{y \in ext\{Y_{n_k}\}} x A(k)) \\ & =(\max\{xA(1)\},\ldots,\max\{xA(k)\}), \\ \end{array}
 $$ where $\max\{xA(\ell)\}$ denotes the maximum component of the vector $xA(\ell)$. Note that this is true because $ext\{Y_{n_\ell}\}$ (the set of extreme points of $Y_{n_\ell}$) consists of the vectors whose $i^{th}$ component is one and the rest is zero, for $i=1,...,n_\ell$.
Analogously,
$$v_I(x, \overline{A}) = (\max \{x\overline{A}(1)\},\ldots,\max
\{x\overline{A}(k)\}).$$

Now, since $x\ge 0$ then $xA({\ell})\ge x\bar{A}({\ell})$, for all
${\ell}=1,\ldots,k$. Hence, $v_I(x, A)\geq v_I(x, \overline{A})$, for
each $x$ and the conclusion follows.

\item [A.2 Column dominance]  Let $A_{c({\ell})}({\ell})$ be the matrix that results from adding to
$A({\ell})$ a new column $H$ that is dominated by a convex combination of
the columns of $A({\ell})$, i.e., $A_{c({\ell})}({\ell})=(A^{\cdot
1}({\ell}),\ldots,A^{\cdot n_{\ell}}(k),H)$ where $H$ is such that there exists
$\alpha_1,\ldots,\alpha_{n_{\ell}} \geq 0, \; \sum_{j=1}^{n_{\ell}}
\alpha_j=1$ satisfying $\sum_{{\ell}=1}^k \alpha_j A^{.j}({\ell}) \geq H. $ It
is
clear that  $A_{c({\ell})}(s)=A(s)$,  $s\neq {\ell}$.

By construction,
$$\hspace*{-1cm}v_I^{\ell}(x, A_{c({\ell})}) = \max (
xA^{\cdot 1}_{c({\ell})}({\ell}),\ldots,xA^{\cdot n_{\ell}}_{c({\ell})}({\ell}), xH)=\max (
xA^{\cdot 1}({\ell}),\ldots,xA^{\cdot n_{\ell}}({\ell}))=v_I^{\ell}(x,A).$$ Then,
$v_I^s(x, A_{c({\ell})}) = v_I^s(x, A)$ for any $s=1,\ldots,k$ and thus,
$v_I(x, A_{c({\ell})}) = v_I(x, A).$

\item [A.4 Row dominance] Let $ A_r = \left(  \begin{array}{ccc}
   & A&  \\
  b_1 & \ldots & b_k
\end{array}  \right) $, $ b_{\ell}=(b_{\ell}^1,\ldots,b_{\ell}^{n_{\ell}}) \in \mathbb{R}^{n_{\ell}},\,
{\ell}=1,\ldots,k$  such that $b_{\ell} = \sum_{i=1}^m \alpha_i A^{i\cdot}({\ell})$
with  $\sum_{i=1}^m \alpha_i=1$ and $\alpha_i\ge 0$. Take
$\overline{x} \in X^1_{m+1}$, then
$$v_I(\overline{x}, A_r) = (\max \overline{x}A_r(1),\ldots,\max
\overline{x}A_r(k)).$$
\par \noindent The security level of $A_r$ in the ${\ell}$-th
component is
\begin{eqnarray*} v^{\ell}_I(\overline{x}, A_r) &= & \max (
\sum_{i=1}^m \overline{x}_i A^{\cdot 1}({\ell})+
\overline{x}_{m+1}b^1_{\ell},\ldots, \sum_{i=1}^m \overline{x}_i A^{\cdot
n_{\ell}}({\ell})+
  \overline{x}_{m+1}b^{n_{\ell}}_{\ell})\\ &=&
 \max ( \sum_{i=1}^m (\overline{x}_i+ \alpha_i \overline{x}_{m+1}
)A^{\cdot 1}({\ell})  ,\ldots, \sum_{i=1}^m (\overline{x}_i+ \alpha_i
\overline{x}_{m+1} )A^{\cdot {n_{\ell}}}({\ell}) )\\
& = & v_I^{\ell}(\widehat{x},A),
\end{eqnarray*}
where $\widehat{x} = (\overline{x}_1+\alpha_1
\overline{x}_{m+1},\ldots,\overline{x}_m+\alpha_m
\overline{x}_{m+1}) \in X_m.$

Then for any $\overline{x}\in X_{m+1}, \exists \widehat{x} \in
X_m$, such that $$ v^{\ell}_I(\overline{x}, A_r)= v^{\ell}_I(\widehat{x},
A_r), \forall {\ell}$$ and conversely.

\item [A.6 Consistency] Assume A.6 is not satisfied. Therefore, there
exists  $z=(z_1,\ldots,z_k) \in VPOSS_k(A)$ such that for all
$0<\alpha<1$, $(\alpha z_1+ (1-\alpha) z_2,z_3,\ldots,z_k) \notin
VPOSS_{k-1}(M[\alpha A(1), (1-\alpha) A(2)], A(3),\ldots,A(k)).$
Hence by Theorem \ref{multiples}, $\nexists
(\beta_1,\beta_3,\ldots,\beta_{k})\in X^>_{k-1}$ such that
{\scriptsize $$  \beta_1\alpha z_1 + \beta_1 (1-\alpha) z_2 + \beta_3z_3+\ldots+ \beta_k z_k =
val (M[\beta_1M[\alpha A(1), (1-\alpha) A(2)],\beta_3
A(3),\ldots,
 \beta_k A(k)]).$$}
 However, this contradicts that $z$ is a payoff vector of a POSS, since by Theorem \ref{t:Voorn} there must exist $(\lambda_1,\ldots,\lambda_k)\in X^>_k$,
 such that $$\sum_{{\ell}=1}^k \lambda_{\ell} z_{\ell} = val (M[ \lambda_1 A(1),\ldots,\lambda_k A(k)]).$$

 Thus, $VPOSS_k$ satisfies A.6.
\end{description}

To finish the proof, it is enough to show that that if $\{f_k\}_{k\ge 1}$ is an arbitrary family  of point-to-set maps that satisfy A.0, A.1, A.2, A.4, and A.6, then for any general multicriteria game given by the matrix $A=(A(\ell))_{\ell=1\ldots k}$ with $A(\ell) \in \mathbb{R}^{m\times n_{\ell}}$ we get $f_k(A) \subseteq VPOSS_k(A) \quad
\forall\; k\geq 1.$


Indeed, for $k=1$, the axioms A.0, A.1, A.2 and A.4 characterize the $val(.)$
 function of a matrix game (see Carpente et al, \cite[Theorem 2]{car}). Therefore, $f_1(.)= val(.)$.

 Let $z \in f_k(A), k\ge 2$. Apply A.6 $k$-times to conclude that there exists
 $\alpha\in \mathbb{R}^k$, $\sum_{{\ell}=1}^k \alpha_{\ell}=1,\; \alpha_{\ell}>0$, such that

 $$\sum_{{\ell}=1}^k \alpha_{\ell} z_{\ell} = val ( M[\alpha_1 A(1),\ldots,\alpha_k A(k)]).$$
Then, by Theorem \ref{t:Voorn}, $z$ is a payoff vector of a POSS
of the multi-matrix $A$. Hence, $$ f_k(A) \subseteq VPOSS_k(A) \quad
\forall A,\; k\geq 1.$$

\hfill$\Box$

After a careful reading of the proof one realizes that an
alternative characterization is still possible using weaker
versions of properties A.0-A.4.

\begin{corollary} \label{c:main}
If the properties A.0, A.1 ,A.2 and A.4 are required only for $k=1$, the characterization of Theorem
\ref{main} still holds.
\end{corollary}

{This corollary implies that $VPOSS _k$ is the largest map on $\mathcal D_2$ that satisfies consistency (A6) and that coincides with the value function on standard single criterion matrix games.
Another characterization of Pareto-optimal
security level vectors, similar to the one in Theorem \ref{t:ccMm}, is also is possible based on
Carpente et al. \cite[Theorem 3]{car}.
}

\begin{theorem} \label{t:cmain}
The set of Pareto-optimal security level vectors $VPOSS_k$ is the largest (w.r.t. inclusion) map on $\mathcal D_2$ that satisfies
 objectivity,   column dominance, row
dominance, column elimination, row elimination and consistency.
\end{theorem}


Comparing theorems \ref{t:cMm}  and \ref{t:cmain} we realize that Pareto-optimal security level vectors  and extended minimax payoff vectors  differ only in the way in which solution payoff vectors from $k$-criteria games are transformed into solution payoff vectors of $(k-1)$-criteria games (consistency properties). The former requires to amalgamate strategies in a game with lower dimension whereas the latter requires to do convex combinations of payoff matrices.

\section{Pairwise logical independence of the properties.\label{sec: independence}}

\medskip

This section shows that the previously presented characterizations use {pairwise} logically independent properties. In doing that we will use some results from the literature plus three additional evaluation maps.

First of all, we observe that since properties $A0-A7$ must hold for any $k\ge 1$ it is enough to show that there exist evaluation maps, for particular choices of $k$, fulfilling only some of these properties. In most cases, this is possible for $k=1$. We introduce three evaluation maps $h_0,\; h_1,\; h_2$, defined on  $G_1=\bigcup_{\scriptsize m,p\in \mathbb{N}} \mathbb{R}^{m\times p}$, the space of the scalar zero-sum games, $h_i:G_1 \rightarrow \mathbb{R}$ for $i=0,1,2$, such that for any $B\in G_1$:
\begin{eqnarray}
h_0(B) & = & 0,\label{l:g0}\\
h_1(B) & = & b_{11},\label{l:g1}\\
h_2(B) & =& \min_{1\le i\le row(B)} b_{i1}, \label{l:g2}
\end{eqnarray}
where $row(B)$ is the number of rows of the matrix $B$.

\underline{\textbf{A0}:}
First of all, to show that objectivity (A0) is logically independent of the rest of properties we use the null function $h_0$. For $k=1$, \cite{norde04} proves that the null function satisfies monotonicity (A1), row dominance (A4) and row elimination (A5). It is also easy to see that $h_0$ also satisfies column elimination ($A3$). However, it does not satisfy objectivity (A0). In addition, \cite{vil} proves using its function $f_4$ that objectivity (A0) and column dominance (A2) are independent.

\underline{\textbf{A1}:}
For $k=1$, the independence of monotonicity (A1) from column dominance (A2) follows from \cite[Theorem 4]{vil}; from column elimination (A3) easily follows using our function $h_1$ and from row dominance (A4) and row elimination (A5) it is proved in \cite{norde04} using its function $f_6$.

\underline{\textbf{A2}:}
Our function $h_1$ satisfies column dominance (A2) but does not verify column elimination (A3) nor row elimination (A5); whereas Theorem 4 in \cite{vil} proves that column dominance (A2) and row dominance (A4) are independent.

\underline{\textbf{A3}:}
Our function $h_1$ shows that column elimination (A3) and row dominance (A4) are independent. The function $h_2$ satisfies row elimination (A5) but does not satisfy column elimination (A3). Moreover, \cite{norde04} shows with its function $f_4$ that row dominance (A4) and column elimination (A3) are independent.

\underline{\textbf{A6,A7}:}
Finally, to prove that objectivity (A0), monotonicity (A1), column dominance (A2), column elimination (A3), row dominance (A4) and  row elimination (A5)  are independent of consistency (A6) and linear consistency (A7) we use the correspondences $\VMm _k$ and $VPOSS _k$, respectively. Examples \ref{ejemplo} and \ref{ex:cont} show that $\VMm _k$ does not satisfy Consistency (A6) and $VPOSS_ k$ does not satisfy Linear Consistency (A7).


Carpente et al. \cite{car} show that the minimax value satisfies
axioms $A0-A5$, although, as shown in the next example, it does not
satisfy  \textit{Consistency}.

\begin{example}\label{ejemplo}
Consider the $2$-criteria game defined by the matrices $ A(1)=
(1,0)$, $A(2)= (0,1)$. (Note that player I only has one pure
strategy while player II has two pure strategies in this
$2$-criteria game.) The minimax values of this game are $(\alpha,1-\alpha),\; \forall \alpha\in [0,1]$. The reader may note that these values do not satisfy
\textit{Consistence} ($A.6$). Indeed,

 $$M\big[\alpha A(1),(1-\alpha)A(2)\big] = \big(\alpha,1,0,1-\alpha\big).$$

It is clear that the minimax value of the single criterion game with
the above matrix {\bf is $1$}. Hence, it can not be a convex combination
of $(\alpha,1-\alpha)$, the  minimax values of the original $2$-criteria game, for any $\alpha \in (0,1)$.\\
\end{example}
\bigskip

Finally, \textit{Linear consistency} is not satisfied by the
Pareto-optimal security payoffs, as shown in the next example.
\medskip

\begin{example}(\ref{ejemplo} continued)\par \label{ex:cont}

Consider the game given in Example \ref{ejemplo}, described by the
two matrices $ A(1)= (1,0)$, $A(2)= (0,1)$. {For this game the unique
Pareto-optimal security payoff is $(1,1)$.} Let us now consider the
game given by the matrix:
$$\alpha A(1)+(1-\alpha)A(2) = (\alpha,1-\alpha).$$
with $1/2 < \alpha\leq 1$. The value of this game is
$\alpha$. Therefore, the security payoff does not satisfy the property
of \textit{linear consistency} for any $\alpha\in (0,1)$,  because
$\alpha$ cannot be obtained as a convex combination of $(1,1)$.\\

The
reader may note that the minimax payoff $(1,0)$ does satisfy this property for $\alpha \in [1/2,1]$.
\end{example}

Table \ref{table:indep} summarizes the pairwise logical independence of properties.
\begin{center}
\begin{table}[h]
\scriptsize
\begin{tabular}{|c|c|c|c|c|c|c|c|c|}
\hline
& A0 & A1 & A2& A3 & A4 &A5 & A6 &A7 \\ \hline
A0 &  X & $h_0$ (\ref{l:g0}) & $f_4$ in \cite{vil} & $h_0$ (\ref{l:g0}) & $h_0$ (\ref{l:g0})& $h_0$ (\ref{l:g0}) & $\VMm _k$ & $VPOSS_k$ \\
\hline
A1 & & X & Th. 4 in \cite{vil} & $h_1$ (\ref{l:g1}) & $f_6$ in \cite{norde04}& $f_6$ in \cite{norde04}& $\VMm _k$ & $VPOSS_k$ \\
\hline
A2 & & & X & $h_1$ (\ref{l:g1})& Th. 4 in \cite{vil} & $h_1$ (\ref{l:g1}) & $\VMm _k$ & $VPOSS_k$ \\
\hline
A3 & & & & X & $h_1$ (\ref{l:g1}) & $h_2$ (\ref{l:g2})& $\VMm _k$ & $VPOSS_k$ \\
\hline
A4 & & & & & X & $f_4$ in \cite{norde04}& $\VMm _k$ & $VPOSS_k$ \\
\hline
A5 & & & & & & X & $\VMm _k$ & $VPOSS_k$ \\
\hline
A6 & & & & & & & X & $VPOSS_k$ \\
\hline
A7 & & & & & & & & X \\ \hline
\end{tabular}
\caption{Pairwise logical independence of properties \label{table:indep}}
\end{table}
\end{center}

\section{Relationship between minimax and POSS in  $\mathcal{D}_1$ and $\mathcal{D}_2$\label{s:6}}
{In this section we show the similarities between minimax and POSS by proving that, when one game in $\mathcal{D}_2$ is transformed into a game in $\mathcal{D}_1$, then the corresponding POSS transforms into minimax.}

In order to introduce our next result we need some further notation. Associated with any matrix
$A=(A(1),\ldots,A(k))$ with $A(\ell)\in \mathbb{R}^{m\times n_{\ell}}$ for all $\ell=1,\ldots,k$, let $AM(A(1)^{\cdot j_1},\ldots,A(k)^{\cdot j_k})\in \mathbb{R}^{m\times k}$ be the matrix that results by joining  the columns $A(\ell)^{\cdot j_{\ell}}$ in $A(\ell)$, for all $\ell=1,\dots,k$. We observe that there are $\prod_{\ell=1}^k n_{\ell}$ different such matrices depending on the different choices of columns in the matrices of $A$.

Next, let us denote as $EM(A(1),\ldots,A(k))\in \mathbb{R}^{m\times k\times \prod_{\ell=1}^k n_{\ell}}$ the matrix that results by joining all the different $AM(A(1)^{\cdot j_1},A(2)^{\cdot j_2},\ldots,A(k)^{\cdot j_k})$ for all $j_1=1,\ldots,n_1,\ldots, j_k=1,\ldots,n_k$. Now, it is clear that
$$\scriptsize EM(A(1),\ldots,A(k))=\big(AM(A(1)^{\cdot 1},\ldots,A(k)^{\cdot 1}),\stackrel{(\prod_{\ell=1}^k n_{\ell})}{\ldots},AM(A(1)^{\cdot n_1},\ldots,A(k)^{\cdot n_k})\big).$$

\begin{theorem}\label{th: transf}
A strategy $x^*$ is a $POSS$ for player $I$ in the  $k$-criteria matrix game  $A=(A(1),\ldots,A(k))\in \mathcal{D}_2$ with $A(\ell)\in \mathbb{R}^{m\times n_{\ell}}$ for all $\ell=1,\ldots,k$ if and only if there exists $\alpha(x^*)  \in X^>_k$ such that $(x^*,\alpha(x^*))$ is an extended minimax strategy in the $(\prod_{\ell=1}^k n_{\ell})$-criteria matrix game $EM(A(1),\ldots,A(k))\in \mathbb{R}^{m\times \prod_{\ell=1}^k n_{\ell}\times k}$.
\end{theorem}
\begin{proof}
Let $\alpha^t=(\alpha_1,\ldots,\alpha_k)\in \mathbb{R}^k$ such that $\alpha_{\ell}\ge 0$ and $\sum_{\ell=1}^k \alpha_{\ell}=1$; and   $\beta^t=(\beta_{j_1\ldots j_k})_{j_1\ldots j_k}\in \mathbb{R}^{\prod_{\ell=1}^k n_{\ell}}$ such that $\beta_{j_1\ldots j_k}\ge 0$ and $\sum_{j_1\ldots j_k} \beta_{j_1\ldots j_k}=1$.
We observe that $EM(A(1),\ldots,A(k))\beta=\sum_{j_1\ldots j_k} \beta_{j_1\ldots j_k} AM(A(1)^{\cdot j_1},\ldots,A(k)^{\cdot j_k}) \in \mathbb{R}^{m\times k}$, and $M(\alpha_1 A(1),\ldots,\alpha_k M(k)) \beta= \sum_{j_1\ldots j_k} \beta_{j_1\ldots j_k} \sum_{\ell=1}^k \alpha_{\ell} A(\ell)^{\cdot j_{\ell}} \in \mathbb{R}^m.$ Therefore,
\begin{eqnarray*} (EM((A(1),\ldots,EM(k))\beta) \alpha&=&\sum_{\ell=1}^k \alpha_{\ell} \sum_{j_1\ldots j_k} \beta_{j_1\ldots j_k}AM(A(1)^{\cdot j_1},\ldots,A(k)^{\cdot j_k})   ,\\
& =& \sum_{j_1\ldots j_k} \beta_{j_1\ldots j_k} \sum_{\ell=1}^k \alpha_{\ell} A(\ell)^{\cdot j_{\ell}}\\
&=& M(\alpha_1 A(1),\ldots,\alpha_k M(k)) \beta.
\end{eqnarray*}
Now, we know that $x^*$ is a POSS for player I if and only if it exists $\alpha^t=(\alpha_1,\ldots,\alpha_k)\in \mathbb{R}^k$ satisfying $\alpha_{\ell}\ge 0$ and $\sum_{\ell=1}^k \alpha_{\ell}=1$ such that $x^*$ is a minimax strategy for single criterion game $M(\alpha_1 A(1),\ldots,\alpha_k M(k))$. This is equivalent to satisfy  that
$$ \max_{\beta } (x^*)^t M(\alpha_1 A(1),\ldots,\alpha_k A(k)) \beta = \min_{x\in X^{>}_m}\max_{\beta}x^t M(\alpha_1 A(1),\ldots,\alpha_k A(k)) \beta.$$
Let $\beta^*$ be the element where the above maximum is attained.
In particular
$$ {\scriptsize  (x^*)^t M(\alpha_1 A(1),\ldots,\alpha_k A(k)) \beta^* =  \min_{x\in X^{>}_m} x^t M(\alpha_1 A(1),\ldots,\alpha_k A(k)) \beta^*. } $$
This is equivalent to
{\scriptsize
\begin{eqnarray*} \sum_{i=1}^m x_i^* \sum_{j_1\ldots j_k} \beta_{j_1\ldots j_k}^* \sum_{\ell=1}^k \alpha_{\ell} A(\ell)^{\cdot j_{\ell}}&=& \min_{x\in X^{>}_m}    \sum_{i=1}^m x_i \sum_{j_1\ldots j_k} \beta_{j_1\ldots j_k}^* \sum_{\ell=1}^k \alpha_{\ell} A(\ell)^{\cdot j_{\ell}} \\
&=&  \min_{x\in X^{>}_m}   \sum_{i=1}^m x_i \sum_{\ell=1}^k \alpha_{\ell} \sum_{j_1\ldots j_k} \beta_{j_1\ldots j_k}^* AM(A(1)^{\cdot j_1},\ldots,A(k)^{\cdot j_k}) \\
\left(\begin{array}{c}\mbox{by scalar minimax theorem applied to matrix}\\ \sum_{j_1\ldots j_k} \beta_{j_1\ldots j_k}^*AM(A(1)^{\cdot j_1},\ldots,A(k)^{\cdot j_k}) \end{array}\right) &=& \min_{x\in X^{>}_m}     \sum_{i=1}^m x_i \sum_{\ell=1}^k \alpha_{\ell} \sum_{j_1\ldots j_k} \beta_{j_1\ldots j_k}^* AM(A(1)^{\cdot j_1},\ldots,A(k)^{\cdot j_k})
\end{eqnarray*}
}
This last equation means that $x^*$ is also a minimax strategy for the single criterion game with matrix $\sum_{j_1\ldots j_k} \beta_{j_1\ldots j_k}^*AM(A(1)^{\cdot j_1},\ldots,A(k)^{\cdot j_k})\in \mathbb{R}^{m\times k}$ where $\beta^*$ is the element where the above maximum is attained. This last matrix is the convex combination with coefficient $\beta^*$ of the matrices $AM(A(1)^{\cdot j_1},\ldots,A(k)^{\cdot j_k})$ therefore by applying Theorem \ref{new}, it follows that $x^*$ is also an extended minimax strategy of the multicriteria game with matrix $$\big(AM(A(1)^{\cdot 1},\ldots,A(k)^{\cdot 1}),\stackrel{(\prod_{\ell=1}^k n_{\ell})}{\ldots},AM(A(1)^{\cdot n_1},\ldots,A(k)^{\cdot n_k})\big).$$
This matrix is by definition $EM(A(1),\ldots,A(k))$ which concludes the proof.
\end{proof}

The next results shows another transformation of a game in $\mathcal{D}_2$ into a game in $\mathcal{D}_1$.

\begin{theorem}\label{th: transf_fed}
Every game $A \in \mathcal{D}_2$ can be transformed into a game $\bar A \in \mathcal{D}_1$. Besides, given $(x,y)$ strategies in $A$ for players I and II, there exists $(\bar x,\bar y)$ strategies for I and II in $\bar A$ such that $xAy = \bar x \bar A \bar y$.
\end{theorem}

\begin{proof}
Given is $A = (A(1),...,A(k))$, with $A(\ell) \in \mathbb{R}^{m \times n_{\ell}}$, which defines a game in $\mathcal{D}_2$. The strategy sets for players I and II are as defined before: $X_m, Y_{n_1,...,n_k}$.

Let us define $\bar A=(\bar A(1),...,\bar A(k))$, with $\bar A(\ell) \in \mathbb{R}^{m \times \prod_{\ell} n_{\ell}}$, which defines a game in $\mathcal{D}_1$, built in such a way that $\bar A(\ell)_{i,(j_1,...,j_k)} = A(\ell)_{ ij_{\ell}}, \ \forall \ i=1,...,m, j_{\ell} = 1,...,n_{\ell}, \ell = 1,...,k$. The strategy sets for players I and II are $X_m$ (the same as in game $A$) and $\bar Y_{\prod} = \{y \in \mathbb{R}^{\prod_{\ell} n_{\ell}} : \sum_{(j_1,...,j_k)} y_{(j_1,...,j_k)} = 1, y \geq 0 \}$.

Given two strategies $x=(x_1,...,x_m) \in X_m$ and $y=(y(1),...,y(k)) \in Y_{n_1,...,n_k}$, define $\bar x = x$ and $\bar y: \bar y_{(j_1,...,j_k)} = \prod_{\ell} y(\ell)_{j_{\ell}}$.

\begin{itemize}
\item Clearly, $\bar x \in X_m$. Let us see that $\bar y \in \bar Y_{\prod}$.

$\sum_{(j_1,...,j_k)} \bar y_{(j_1,...,j_k)} = \sum_{(j_1,...,j_k)} \prod_{\ell} y(\ell)_{j_{\ell}} = \prod_{\ell}  \sum_{j_{\ell}} y(\ell)_{j_{\ell}} = 1.$

\item Now, let us check that $xAy = \bar x \bar A \bar y$. For this purpose, we need to prove that $xA(\ell)y(\ell) = \bar x \bar A(\ell) \bar y, \ \forall \ \ell = 1,...,k$.

    Take $\ell \in \{1,...,k\}$. Because $xA(\ell)y(\ell) = \sum_{i} x_i (A(\ell) y(\ell))_i$, and $\bar x \bar A(\ell) \bar y = \sum_{i} \bar x_i (\bar A(\ell) \bar y)_i= \sum_{i} x_i (\bar A(\ell) \bar y)_i$, the proof of this statement reduces to check that $(A(\ell) y(\ell))_i = (\bar A(\ell) \bar y)_i, \ \forall \ i=1,...,m.$

\begin{enumerate}
\item    $(A(\ell) y(\ell))_i = \sum_{j_{\ell}} A(\ell)_{i,j_{\ell}} y(\ell)_{j_{\ell}}$.
\item   $$
        \begin{array}{rl}
        (\bar A(\ell) \bar y)_i & = \sum_{(j_1,..,j_k)} \bar A(\ell)_{i,(j_1,...,j_k)} \bar y_{(j_1,...,j_k)} \\ & = \sum_{(j_1,..,j_k)} A(\ell)_{i,j_{\ell}} \prod_{\ell'}y(\ell')_{j_{\ell'}} \\
        & =\sum_{j_{\ell}}( \sum_{J \setminus j_{\ell}} A(\ell)_{i,j_{\ell}}\prod_{\ell'}y(\ell')_{j_{\ell'}}) \\
        &=\sum_{j_{\ell}}A(\ell)_{i,j_{\ell}} y(\ell)_{j_{\ell }}(\sum_{J \setminus j_{\ell}} \prod_{\ell' \neq \ell }y(\ell')_{j_{\ell'}}) \\
        &= \sum_{j_{\ell}}A(\ell)_{i,j_{\ell}} y(\ell)_{j_{\ell }}(\prod_{\ell' \neq \ell }\sum_{ j_{\ell'}} y(\ell')_{j_{\ell'}}) \\
        & =    \sum_{j_{\ell}}A(\ell)_{i,j_{\ell}} y(\ell)_{j_{\ell }}=(A(\ell) y(\ell))_i.
            \end{array},$$ where $J$ is sometimes used to denote the complete vector of indexes $(j_1,...,j_k)$.
\end{enumerate}
This proves that $(\bar A(\ell) \bar y)_i = (A(\ell) y(\ell))_i \ \forall \ \ell, i$, and therefore $xAy = \bar x \bar A \bar y$.
\end{itemize}
\end{proof}

\begin{example}Applying the transformation in Theorem \ref{th: transf_fed} to the game in Example \ref{ex:gameD2}, the new payoff matrix is:
$$  \bar A = \left(
  \begin{array}{cccccc}
(0,-2) & (0,0) & (0,0.5) & ({-1},-2)  & ({-1},0) & ({-1},0.5) \\
(1,-1) & (1,0) & (1,1) & (0,-1) & (0,0) & (0,1) \\
  \end{array}
\right) ,
$$
If we write the same matrix in a similar way as we did for games in $\mathcal{D}_2$, we have
$ \bar A = (\bar A(1),\bar A(2))$, with:
$$  (\bar A(1),\bar A(2)) =
\left(
  \begin{array}{cc}
    \left(
      \begin{array}{cccccc}
        0 & \ 0 & \ 0 & \ -1 & \ -1 & \ -1\\
        1 & \ 1 & \ 1 &  \ 0 & \  0 & \  0 \\
      \end{array}
    \right) ,
     &
     \left(
       \begin{array}{cccccc}
         -2 & \ 0 & \ 0.5 & \ -2 & \ 0 & \ 0.5  \\
         -1 & \ 0 & \ 1 & \ -1 & \ 0 & \ 1\\
       \end{array}
     \right)
     \\
  \end{array}
\right)
$$
Note that this new game has two criteria, two strategies for player I, and six ($3 \times 2$) strategies for player II. The strategy sets for the players are $X_2$ and $Y_6$.

If we apply to the original game the transformation in Theorem \ref{th: transf} we get:
$EM(A(1),A(2)) = A_j = (A_j(1), ... , A_j(6))$, where $A_j(1)$ ... $A_j(6)$ are:
$$
\left(
  \begin{array}{cccccc}
    \left(
      \begin{array}{cc}
        0 & \ -2\\
        1 & \  -1\\
      \end{array}
    \right) ,
     &
    \left(
      \begin{array}{cc}
        0 & \ 0\\
        1 & \ 0\\
      \end{array}
    \right) ,
     &
    \left(
      \begin{array}{cc}
        0 & \ 0.5\\
        1 & \ 1\\
      \end{array}
    \right) ,
     &
     \left(
       \begin{array}{cc}
         -1 & \ -2  \\
         0 & \ 1 \\
       \end{array}
     \right) ,
     &
     \left(
       \begin{array}{cc}
         -1 & \ 0  \\
         0 & \ 0 \\
       \end{array}
     \right) ,
          &
     \left(
       \begin{array}{cc}
         -1 & \ 0.5  \\
         0 & \ 1 \\
       \end{array}
     \right)
     \\
  \end{array}
\right)
$$
\end{example}

Let us denote by $\bar{\mathcal{D}}_1$ the class of games in $\mathcal{D}_1$ that can be obtained from a game in $\mathcal{D}_2$ as in Theorem \ref{th: transf_fed}. The following lemma states that strategies in the game in $\bar{ \mathcal{D}}_1$ can be transformed into strategies in the corresponding game in $\mathcal{D}_2$ keeping the same payoffs.

\begin{lemma}\label{lemm:D1toD2}
Let $A$ be a game in ${\mathcal{D}}_2$, and let $\bar A$ be its corresponding transformation into a game in $\bar{\mathcal{D}}_1$. Let $(x, \bar y)$ be a pair of mixed strategies for players I and II in game $\bar A$. Then, there exists a pair of strategies $(x,y)$ for the game $A$ such that $x \bar A \bar y = x A y$.
\end{lemma}
\begin{proof}
Let $(x,\bar y)  \in X_m \times Y_{\prod}$ be a pair of strategies for game $\bar A$. Build $y=(y(1),...,y(k))$ so $y(\ell)_{j} = \sum_{(j_1,...,j_k) : j_{\ell} = j} \bar{y}_{j_1,...,j_k}$, for  $j=1,...,n_{\ell}, \ell = 1,...,k$. Clearly, $\sum_{j_{\ell}} y(\ell)_{j_{\ell}} = 1$, and all such components are positive, therefore $y \in Y_{n_1,...,n_k}$.

Take one of the criteria, $\ell$. We have that $x \bar{A}(\ell) \bar y = \sum_i \sum_{J} x_i \bar{A}(\ell)_{i,J} \bar{y}_J$, and $x {A}(\ell)  y = \sum_i \sum_{j_{\ell}} x_i {A}(\ell)_{i,j_{\ell}} {y}_{j_{\ell}}.$

The following equation proves that both payoffs are equal:
$$ \sum_{J} \bar{A}(\ell)_{i,J} \bar{y}_J = \sum_{J} {A}(\ell)_{i,j_{\ell}} \bar{y}_J =   \sum_{j_{\ell}}  {A}(\ell)_{i,j_{\ell}} \sum_{(j_1,...,j_k) : j_{\ell} }\bar{y}_J =  \sum_{j_{\ell}}  {A}(\ell)_{i,j_{\ell}} {y}_{j_{\ell}}.$$
\end{proof}

The following theorem states that POSS strategies in $\mathcal{D}_2$ can be transformed into POSS strategies in $\bar{\mathcal{D}}_1$, and viceversa.

\begin{theorem}
Let $A$ and $\bar A$ be a game in $\mathcal{D}_2$ and its corresponding transformation into a game in $\bar{\mathcal{D}}_1$. Then we have that  $(x^*,y^*)$ is a pair of POSS in game $A$ if and only if $(x^*,\bar{y}^*)$ is a pair of {minimax strategies} in game $\bar A$, {where $\bar y^* \in Y_{\prod}$ is obtained from $y^*$ as in Theorem \ref{th: transf_fed}.}
\end{theorem}
\begin{proof}

\begin{itemize}
\item
Let us first see that $x^*$ is a POSS for I in $A$ if and only if $x^*$ is minimax for I in $\bar A$. For any strategy of {player} II, the worst payoff player I can obtain from $A(\ell)$ is the same  as the worst payoff he/she can obtain from $\bar{A}(\ell)$ (note that the columns of $\bar{A}(\ell)$ are the same as in ${A}(\ell)$, but repeated and in different order). And therefore the concept of POSS for I in $A$, and the concept of minimax for I in $\bar A$ coincide.
\item
Now let us see that $y^*$ is POSS for II in $A$ if and only if $\bar y^*$ is {maximin} for II in $\bar A$. Because the payoffs are the same in both games, that is, for any $x \in X_m$, $x A(\ell) y^* = x \bar A(\ell) \bar y^*$, we have that for any strategy of I, the worst payoff player II can obtain from $A$ is the same as the worst payoff player II can obtain for $\bar A$, and therefore the concept of POSS for II in $A$ coincide with the concept of {maxmin} for II in $\bar A$.

\end{itemize}
\end{proof}

\section{Conclusions}
Our results in sections \ref{sec: poss} and \ref{sec: minmax} show that
multicriteria minimax and Pareto-optimal security payoffs are rather
similar since they can be characterized with similar set of
properties. The main difference between them comes from the different
form of consistency that each solution concepts satisfies. On the one
hand, in the minimax case the consistency requires weighted sum of
matrices, thus reducing the analysis to games with smaller number of
criteria but with the same set of strategies.  On the other hand, in
the Pareto-optimal security case consistency gives rise also to
smaller number of criteria but modifying the set of strategies in
one of the scalar component games. This difference is very important
and it explains the different structure of the corresponding set of
strategies: \PO strategies are well-known to be a connected union of
polyhedral sets (see \cite{fernandez96}) whereas multicriteria
minimax strategies are not, in general, finite union of polytopes
(see \cite{Voorequi}).

\section{Acknowledgments}
We would like to thank Dr. Francisco R. Fernández for his useful comments on  earlier versions of this paper. The authors also want to acknowledge the financial support from grants FQM-5849 (Junta de Andaluc\'ia$\backslash$FEDER) and MTM2010-19576-C02-01, {MTM2013-46962-C02-01} (MICINN, Spain).

\bibliographystyle{alpha}
{}

\end{document}